\documentclass[11pt]{amsart}
\usepackage{amsfonts,amssymb,amsmath,amsthm,here}

\usepackage{color}  
\usepackage{array}
\usepackage{mathrsfs}
\newtheorem{theo}{Theorem}

\newtheorem{cor}{Corollary}
\newtheorem{lem}{Lemma}

\theoremstyle{definition}

\newtheorem{rem}{Remark}

\makeatletter

\@addtoreset{equation}{section}
\makeatother

\makeatletter
\@namedef{subjclassname@2020}{
  \textup{2020} Mathematics Subject Classification}
\makeatother

\def \Z{\mathbb Z}
\def \Q{\mathbb Q}

\def \calO{\mathcal O}
\def \calA{\mathcal A}

\def \calD{\mathcal D}

\allowdisplaybreaks[4]

\author{Miho Aoki }
\address{Department of Mathematics,
Interdisciplinary Faculty of Science and Engineering,
Shimane University,
Matsue, Shimane, 690-8504, Japan}
\email{aoki@riko.shimane-u.ac.jp}

\thanks{This work was supported by JSPS KAKENHI Grant Number JP21K03181} 

\title[Cyclic cubic fields]{Gaussian periods and Shanks' cubic polynomials. II}

\begin{document}

\begin{abstract} 
We give a linear relation between a cubic  Gaussian period and  a root of Shanks' cubic polynomial 
in  wildly ramified cases.
\end{abstract}
\maketitle
\section{Introduction}\label{sec:intro}
Let $L$ be a cyclic cubic field. The conductor $\mathfrak f$ of $L$ should be 
\begin{equation}\label{eq:cond}
\mathfrak f=\begin{cases}
p_1\cdots p_{\nu} & \text{if $3\nmid \mathfrak f $\, (tamely ramified)}, \\
3^2 p_1\cdots p_{\nu} &    \text{if $3 | \mathfrak f$\, (wildly ramified)},
\end{cases}
\end{equation}
where $p_1,\ldots, p_{\nu}$ are different prime numbers satisfying $p_1 \equiv \cdots \equiv p_{\nu}\equiv 1
\pmod{3}$ (\cite[p.\,10]{H}). 
For a positive integer $n$, let $\zeta_n=e^{2\pi i/n}$ be the $n$-th root of unity. We define {\it  the Gaussian periods} $\eta_i\, (i=0,1,2)$  of  $L$ by
\begin{equation}\label{eq:period}
\eta_0={\rm Tr}_{\Q(\zeta_{\mathfrak f})/L} (\zeta_{\mathfrak f}),\quad \eta_1=\sigma (\eta_0), \quad \eta_2=\sigma^2 (\eta_0)
\end{equation}
where ${\rm Tr}_{\Q (\zeta_{\mathfrak f})/L}$ is the trace map from $\Q (\zeta_{\mathfrak f})$ to $L$, and $\sigma$ is a 
generator of ${\rm Gal}(L/\Q)$. It is known that $L=\Q (\eta_i)$ holds for any $i=0,1,2$.
We define {\it the period polynomial} $P(X)\, (\in \Z [X])$ by
\begin{equation}\label{eq:P-poly}
P(X)=(X-\eta_0)(X-\eta_1)(X-\eta_2).
\end{equation}
An explicit formula of $P(X)$ is  known as follows (\cite{G}, \cite{H}, \cite[p.\,90]{Gr}, \cite[p.\,8--9]{M}).
\begin{equation}\label{eq:P}
P(X)=\begin{cases}
X^3-\mu (\mathfrak f)X^2 +\dfrac{1-\frak f}{3} X - \mu (\mathfrak f) \dfrac{ (M-3)\mathfrak f +1}{27}& \text{if $3\nmid \mathfrak f $}, \\
& \\
X^3-\dfrac{\frak f}{3} X - \mu (\mathfrak f/9) \dfrac{ \mathfrak f M}{27}& \text{if $3| \mathfrak f $}, \\
\end{cases}
\end{equation}
where $\mu$ is the M\"{o}bius function, and
$M \, (\in \Z)$ satisfies the following for some $N\, (\in \Z)$.
\begin{equation}\label{eq:4f}
4 \mathfrak f=M^2+27N^2, 
\end{equation}
\[ \text{where}\ 
\begin{cases}
M \equiv 2\pmod{3},\ N>0& \text{if $3\nmid \mathfrak f $}, \\
M=3M_0,\ M_0 \equiv 2\pmod{3},\ N\not\equiv 0 \pmod{3},\ N>0 & \text{if $3| \mathfrak f $}. \\
\end{cases}
\]
On the other hand, for an  integer $\mathfrak f$ given in the form (\ref{eq:cond}), 
there are exactly $2^{\nu-1}$ (resp. $2^{\nu}$) pairs $(M,N)\in \Z \times \Z$   which satisfy (\ref{eq:4f}) 
(\cite[p.\,342-343,\, p.\,364\, Exercise~18]{Co}), and each pair $(M,N)$ corresponds to exactly $2^{\nu-1}$ (resp. $2^{\nu}$)  cyclic cubic fields $\Q(\eta_0)$ with conductor $\mathfrak f$ in the case of tamely (resp. wildly) ramified.

Next, we will explain the known results on the connection between the period polynomial and Shanks' cubic polynomial. 
For $n\in \Q$, we define {\it Shanks' cubic polynomial} $f_n(X)\, (\in \Q[X])$ by
\begin{equation}\label{eq:fn}
f_n(X)=X^3-nX^2-(n+3)X-1.
\end{equation}
The discriminant of $f_n(X)$ is $d(f_n)=(n^2+3n+9)^2$.
For a root of $\rho_n$ of $f_n(X)$, let $L_n=\Q(\rho_n)$ and $G={\rm Gal}(L_n/\Q)$. If $f_n(X)$ is irreducible over $\Q$, then $L_n$ is a cyclic cubic field.
In this case, 
we put $\rho_n'=\sigma (\rho_n)$ and $\rho_n''=\sigma^2 (\rho_n)$ where $\sigma$ is a generator of $G$.
It is known that $f_n(X) $ is a generic cyclic cubic polynomial (\cite[chap.\,1]{Se}). 
Namely, for any cyclic cubic field $L$, there exists $n\in \Q$ such that $L=L_n$.
If $n\in \Z$, then $f_n(X)$ is always irreducible, and the field $L_n$ is known as
{\it the simplest cubic field} (see \cite{S}, \cite{W}).

As explained below, in special cases, the relation between Shanks' polynomial $f_n(X)$ and the period polynomial $P(X)$ is known.
Let $\mathfrak f =p_1\cdots p_{\nu}$ be an integer where $p_1,\ldots,p_{\nu}$ are different prime numbers which satisfy
$p_1\equiv \cdots \equiv p_{\nu} \equiv 1\pmod{3}$, and a pair $(M,N) \in \Z \times \Z$ satisfies (\ref{eq:4f}).
Assume that $N=1$. In this case, it is known that  $L=\Q (\eta_0)$ is 
a simplest cubic field $L_n$ for  $n=\left(M-3\right)/2 \in \Z$, and 
there is a linear relation between the Gaussian period $\eta_i$ and  a root of $f_n(X)$
(\cite[p.\,536]{Le}, \cite{Ch}, \cite[Proposition~2.2]{La}).
 We can easily check $4\mathfrak f=M^2+27=4(n^2+3n+9)$,
and hence $\mathfrak f=n^2+3n+9$. If we use the explicit formula (\ref{eq:P}) of the period polynomial, we can check
\begin{equation}\label{eq:f-P-N=1}
\mu (\mathfrak f)  f_n (X) =P\left(\mu (\mathfrak f) (X+v_n )\right)
\end{equation}
where $v_n=\left( 1-n\right)/3 \, (\in \Z)$, and hence we obtain
\begin{equation}\label{eq:eta-rho-N=1}
\{ \eta_0,\eta_1,\eta_2 \} =\left\{ \mu (\mathfrak f)  (\rho_n +v_n ), \mu (\mathfrak f)  (\rho_n'+v_n), 
\mu(\mathfrak f) (\rho_n''+v_n) \right\}
\end{equation}
and $L=L_n$.

In this paper, we will extend these results (\ref{eq:f-P-N=1}) and (\ref{eq:eta-rho-N=1}) for general pairs $(M,N)$
(not necessarily $N=1$) of (\ref{eq:4f}) without the explicit formula (\ref{eq:P}) of $P(X)$ in the case of wildly ramified.
In the case of tamely ramified, the author gave the following theorem.
\begin{theo}[\cite{A2}]\label{theo:pre-main}
Let $\mathfrak f =p_1\cdots p_{\nu}$ be an integer where $p_1,\ldots,p_{\nu}$ are different prime numbers which satisfy
$p_1\equiv \cdots \equiv p_{\nu} \equiv 1\pmod{3}$, and a pair $(M,N) \in \Z \times \Z$ satisfies (\ref{eq:4f}).
Put $n_1=(M-3N)/2,\ n_2=N$ and $n=n_1/n_2$.
Then $n_1$ and $n_2$ satisfy the following.
\begin{enumerate}
\item[(1)] $n_1$ and $n_2$ are coprime integers.
\item[(2)] $n_1^2+3n_1n_2+9n_2^2=\mathfrak f $.
\item[(3)] $f_n(X)$ is irreducible and the conductor of the cyclic cubic field of $L_n=\Q (\rho_n)$ is $\mathfrak f$.
\item[(4)] $n_2^3 \mu(\mathfrak f) f_n(X)=P\left(\mu (\mathfrak f)\left(n_2X+\frac{1-n_1}{3}  \right) \right)$ holds, where $P(X)$ is the period polynomial given by 
(\ref{eq:P-poly}) whose roots are the Gaussian periods $\eta_0,\eta_1,\eta_2$ of $L_n$.
\item[(5)] $ \{ \eta_0, \eta_1, \eta_2 \}=\left\{ \mu (\mathfrak f) \left( n_2 \rho_n +\frac{1-n_1}{3} \right),  \mu (\mathfrak f) \left( n_2 \rho'_n +\frac{1-n_1}{3} \right), 
 \mu (\mathfrak f) \left( n_2 \rho''_n +\frac{1-n_1}{3} \right) \right\}$.
\end{enumerate}
Furthermore, all cyclic cubic fields with conductor $\mathfrak f$ are given by $L_n$ for such $n=n_1/n_2$.
\end{theo}

We will give a similar theorem in the case of wildly ramified.
The main result of this paper is as follows.
\begin{theo}\label{theo:main}
Let $\mathfrak f =3^2p_1\cdots p_{\nu}$ be an integer where $p_1,\ldots,p_{\nu}$ are different prime numbers which satisfy
$p_1\equiv \cdots \equiv p_{\nu} \equiv 1\pmod{3}$, and a pair $(M,N) \in \Z \times \Z$ satisfies (\ref{eq:4f}).
Put $n_1=(M-3N)/2,\ n_2=N$ and $n=n_1/n_2$.
Then $n_1$ and $n_2$ satisfy the following.
\begin{enumerate}
\item[(1)] $n_1$ and $n_2$ are coprime integers.
\item[(2)] $n_1^2+3n_1n_2+9n_2^2=\mathfrak f $.
\item[(3)] $f_n(X)$ is irreducible and the conductor of the cyclic cubic field of $L_n=\Q (\rho_n)$ is $\mathfrak f$.
\item[(4)] $n_2^3 \mu(\mathfrak f/9) f_n(X)=P\left(\mu (\mathfrak f/9)\left(n_2X-\frac{n_1}{3}  \right) \right)$ holds, where $P(X)$ is the period polynomial given by 
(\ref{eq:P-poly}) whose roots are the Gaussian periods $\eta_0,\eta_1,\eta_2$ of $L_n$.
\item[(5)] $ \{ \eta_0, \eta_1, \eta_2 \}=\left\{ \mu (\mathfrak f/9) \left( n_2 \rho_n -\frac{n_1}{3} \right),  \mu (\mathfrak f/9) \left( n_2 \rho'_n -\frac{n_1}{3} \right), 
 \mu (\mathfrak f/9) \left( n_2 \rho''_n -\frac{n_1}{3} \right) \right\}$.
\end{enumerate}
Furthermore, all cyclic cubic fields with conductor $\mathfrak f$ are given by $L_n$ for such $n=n_1/n_2$.
\end{theo}

We will prove the theorem without known result (\ref{eq:P}) on  the period polynomial, and use recent results \cite{A1}
on the Galois module structure of the ring of integers of cyclic cubic fields.
\begin{rem}\label{rem-of- theo}
\begin{enumerate}
\item[(1)] If we use the explicit formula  (\ref{eq:P}) of $P(X)$,
then we can easily prove (4) and (5) by direct calculation. Conversely, if we have (4) and (5), then we can obtain the explicit formula (\ref{eq:P}).
\item[(2)] By the theorem, we know that if $L$ is a wildly ramified cubic field, then there exists coprime integers $n_1$ and $n_2$ which satisfy $L=L_n$
for $n=n_1/n_2$ and $(n_1^2+3n_1n_2+9n_2^2)/9$ is square-free.
\item[(3)] Let  $M$ and $N$ be integers satisfying (\ref{eq:4f}) and put 
\begin{equation}\label{eq:char}
\dfrac{M+3N\sqrt{-3}}{2}=
\begin{cases}
\pi_1\cdots \pi_{\nu} &\text{if $3\nmid \mathfrak f$}, \\
3\zeta_3^{\pm 1} \pi_1\cdots \pi_{\nu} &\text{if $3 |\mathfrak f$},
\end{cases}
\end{equation}
where $\pi_i \in \Z [\zeta_3]$ are prime elements which divide $p_i$ and 
$
-\tau (\chi_{p_i})^3 =p_i \pi_i$ for the character $\chi_{p_i}$ defined by $\chi_{p_i}(a) \equiv a^{(p_i-1)/3}\pmod{(\pi_i)}$,
and $\displaystyle{ \tau (\chi_{p_i})=\sum_{a\in (\Z/p_i\Z)^{\times}} \chi_{p_i} (a) \zeta_{p_i}^a }$ is the Gaussian sum.
Let $\chi_{3^2}$ be the character defined by  $\chi_{3^2}(a) =\zeta_3^{\pm(a^2-1)/3}$ (double sign in same order in (\ref{eq:char})). Put
\[
\chi=
\begin{cases}
\chi_{p_1} \cdots \chi_{p_{\nu}} & \text{if $3\nmid \mathfrak f$},\\
\chi_{3^2} \chi_{p_1}\cdots \chi_{p_{\nu}} &   \text{if $3 |\mathfrak f$}.
\end{cases}
\]
Then the cyclic cubic fields $L_n$ of Theorems~\ref{theo:pre-main} and \ref{theo:main} are the fields corresponding to ${\rm Ker}\, \chi  \leq (\Z/\mathfrak f\Z)^{\times}$
(\cite[p.~12--13]{H}).
\end{enumerate}
\end{rem}
\section{Preliminaries}\label{sec:Pre}

In this section, 
we prove some lemmas and a  theorem used in the proof of Theorem~\ref{theo:main}.
\begin{lem}\label{lem:irre}
Let $n=n_1/n_2$ be a rational number where the integers $n_1$ and $n_2$ are coprime. Suppose that $3|n_1, 9 || \Delta_n$ and $\Delta_n/9$ is square-free,
where $\Delta_n=n_1^2+3n_1n_2+9n_2^2$.
Then the cubic polynomial $f_n(X)$ is irreducible over $\Q$.
\end{lem}
\begin{proof}
First, we show that $\Delta_n/9$ and $2n_2+3n_2$ are coprime.
Let $p$ be a prime number which divides both $\Delta_n/9 $ and $2n_1+3n_2$. We can easily check $p\ne 2,3$ since
$2\nmid (\Delta_n/9)$ and $3\nmid (\Delta_n/9)$.
Furthermore, since $4\Delta_n =(2n_1+3n_2)^2+27n_2^2$ and $p\ne 3$, we have $p |n_2$. Hence we have $p |2n_1$ since $p$ divides both $2n_1+3n_2$ and $n_2$.
This is a contradiction since $p\ne 2 $ and $(n_1,n_2)=1$. Therefore $\Delta_n/9$ and $2n_1+3n_2$ are coprime. 
The irreducibility of $f_n(X) $ can be obtained  by using Eisenstein's criterion for the right-hand side of
\[
(3n_2)^3 f_n \left( \frac{X}{3n_2}+\frac{n}{3} \right) =X^3-3\Delta_n X -(2n_1+3n_2)\Delta_n
\]
and a prime factor of $\Delta_n/9$ if $\Delta_n/9 \ne 1$. If $\Delta_n/9 =1$, then we have $n_2=\pm 1$ since
$1=\Delta_n/9 =(n_1/3+n_2/2)^2+3/4\, n_2^2 \geq 3/4\, n_2^2$, and hence $n=n_1/n_2 \in \Z$ and
$f_n(x)$ is irreducible over $\Q$ (in this case, we have $n=0,-3$ and $L_0=L_{-3}$).
\end{proof}
\begin{lem}\label{lem:not-L}
Let $n=n_1/n_2$ and $n'=n_1'/n_2'$ be rational numbers where the integers $n_1, n_2$ and $ n_1', n_2'$ satisfy $(n_1, n_2)=(n_1', n_2')=1$.
Put $\Delta_n=n_1^2+3n_1n_2+9n_2^2$ and $\Delta_{n'}= {n_1'}^2+3n_1'n_2'+9{n_2'}^2$.
Suppose that the following (i) $\sim$ (iii) holds.
\begin{enumerate}
\item[(i)] $3|n_1,\, 3|n_1'$ and $2n_1/3+n_2 \equiv 2n_1'/3+n_2' \equiv 2 \pmod{3}$.
\item[(ii)] $\Delta_n/9$ and $\Delta_{n'}/9$ are square-free, and $9 || \Delta_n,\, 9 ||\Delta_{n'}$.
\item[(iii)] $2n_1+3n_2 \ne 2n_1'+3n_2'$.
\end{enumerate}
Then we have $L_n \ne L_{n'}$.
\end{lem}
\begin{proof}
 We have
\begin{align}
(3n_2)^3 f_n \left(\frac{3X}{n_2} +\frac{n}{3} \right) =3^6 \left(X^3-\dfrac{\Delta_n}{9} \cdot \dfrac{X}{3} -\left( \dfrac{2n_1}{3} +n_2 \right) \dfrac{\Delta_n}{9}\cdot
\dfrac{1}{27} \right),\notag \\
\label{eq:fn-lemma} \\
(3n_2')^3 f_n \left(\frac{3X}{n_2'} +\frac{n'}{3} \right) =3^6 \left(X^3-\dfrac{\Delta_{n'}}{9} \cdot \dfrac{X}{3} -\left( \dfrac{2n'_1}{3} +n'_2 \right) \dfrac{\Delta_{n'}}{9}\cdot
\dfrac{1}{27} \right). \notag 
\end{align}
We can show that all prime numbers $p$ which divides $\Delta_n/9$ or $\Delta_{n'}/9$ satisfy $p\equiv 0,1 \pmod{3}$ (see \cite[Lemma~3]{A1}).
Furthermore, we have $p\ne 3$  from the assumption (ii).
Therefore, both $\Delta_n/9$ and $\Delta_{n'}/9$ are products of distinct prime numbers $p$ satisfying $p\equiv 1 \pmod{3}$. Furthermore, we have
$2n_1/3+n_2 \equiv 2n_1'/3 +n_2' \equiv 2\pmod{3}$ from (i). 
Since $f_n(X)$ and $f_{n'}(X)$ are irreducible over $\Q$ by Lemma~\ref{lem:irre}, both $L_n$ and $L_{n'}$ are cyclic cubic fields.
Using these facts  and \cite[Lemma~6.4.5]{Co}, the roots of the right-hand sides of the two equations of 
(\ref{eq:fn-lemma}) give different cyclic cubic fields, and we have $L_n \ne L_{n'}$.
\end{proof}

Let $L/\Q$ be a finite abelian extension with Galois group $G$. Leopoldt showed that the ring of integers $\calO_L$ of $L$ is a free module
of rank $1$ over the associated order $\calA_{L/\Q}:= \{ x\in \Q [G]\ |\ x\calO_L \subset \calO_L \}$ (\cite[Satz~6]{Leo}. \cite[Theorem~2]{Let}).
The following lemma is a part of a recent result of \cite[Corollary~5]{A1},  which is a generalization of the results
\cite{HA} and \cite{OA} for the simplest cubic field.
\begin{lem}\label{lem:Aoki}
Let $n=n_1/n_2$ be a rational number where the integers $n_1$ and $n_2$ are coprime,
$\mathfrak f$ be the conductor of $L_n$. Suppose that $9 || \Delta_n$ and  $\Delta_n/9$ is square-free,
where $\Delta_n=n_1^2+3n_1n_2+9n_2^2$. Put $\alpha=n_2\rho_n-n_1/3$. Then we have  $\alpha \in {\bf e}_{\mathfrak f} \calO_{L_n}$ for ${\bf e}_{\mathfrak f}=
(2-\sigma -\sigma^2)/3$ and 
 $\alpha+1$ is a generator of $\calO_{L_n}$ over $\calA_{L_n/\Q}$, namely we have $\calO_{L_n}=\calA_{L_n/\Q}(\alpha +1)$
(see \S \ref{sec:Af} for the definition of ${\bf e}_{\mathfrak f} \in \Q[G]$).
\end{lem}

\section{Structure of the units group of the associated order}\label{sec:Af}

Let $p$ be an odd prime number and $L/\Q$ a cyclic extension of degree $p$ with Galois group $G=\langle  \sigma \rangle$.
The conductor $\mathfrak f$ of $L$ should be
\begin{equation}\label{eq:cond-gen}
\mathfrak f=\begin{cases}
p_1\cdots p_{\nu} & \text{if $p\nmid \mathfrak f $\, (tamely ramified)}, \\
p^2 p_1\cdots p_{\nu} &    \text{if $p | \mathfrak f$\, (wildly ramified)},
\end{cases}
\end{equation}
where $p_1,\ldots, p_{\nu}$ are different prime numbers satisfying $p_1 \equiv \cdots \equiv p_{\nu}\equiv 1
\pmod{p}$.
Let $v_p(x)$  denote the $p$-adic valuation of $x \in \Q$ for a prime number $p$.
For any $m\in \mathbb Z_{>0}$, we put
\[
p (m)=\prod_{\substack{ 
p\mid m
\\
p\neq 2
}} p,\qquad 
q(m) = \prod_{\substack{
p \\
v_p(m) \geq 2}} p^{v_p(m)}.
\]
where the first product runs over all odd prime numbers $p$ dividing $m$, and the second product runs over all
prime numbers $p$ that satisfy $v_p (m)\geq 2$. Put
\[
 \calD (\mathfrak f)= \{ m\in \Z_{>0} \, |\, p(\mathfrak f) | m,\ m| \mathfrak f,\ m\not\equiv 2 \pmod{4}\}.
\]
Let $\mathscr X$ be the group of Dirichlet characters associated to $L$.
 We define {\it a branch class} of $\mathscr X$ for any $m\in \calD (\mathfrak f)$ 
by
\[
\Phi_m = \{ \chi \in \mathscr{X}  \, |\, q(\mathfrak f_{\chi})=q(m) \}.
\]
where $\mathfrak f_{\chi}$ is the conductor of $\chi$.
We have $\mathscr{X}=\coprod_{m\in \calD (\mathfrak f)} \Phi_m$ (disjoint union). For any $\chi \in X$,
let
\[
{\bf e}_{\chi} = 
\frac{1}{[L:\Q]} 
\sum_{g \in G} \chi^{-1} (g) g
\]
be the idempotent. Furthermore, for any $m \in \calD (\mathfrak f)$, let
\[
{\bf e}_m = \sum_{\chi \in \Phi_m} {\bf e}_{\chi}.
\]
Since the branch class $\Phi_m$ is closed under conjugation,
we obtain ${\bf e}_m \in \Q [G]$.
Since the conductor $\mathfrak f$ is given by (\ref{eq:cond-gen}) for a cyclic extension $L/\Q$
of degree $p\, (\ne 2)$, we have
\[
D(\mathfrak f)= 
\begin{cases}
\{ \mathfrak f \} & \text{if $p\nmid \mathfrak f $,} \\
\{ \mathfrak f, \mathfrak f/p \}&    \text{if $p | \mathfrak f$}.
\end{cases}
\]
Leopoldt   (\cite[Satz~6]{Leo}. \cite[Theorem~2]{Let}) showed that       
\begin{equation}\label{eq:O_L}
\calO_L= 
\begin{cases}
\calA_{L/\Q} {\rm Tr}_{\Q(\zeta_{\mathfrak f})/L}(\zeta_{\mathfrak f}) & \text{if $p\nmid \mathfrak f $,} \\
\calA_{L/\Q} ({\rm Tr}_{\Q(\zeta_{\mathfrak f})/L}(\zeta_{\mathfrak f})+1) &    \text{if $p | \mathfrak f$}.
\end{cases}
\end{equation}
and
\begin{equation}\label{eq:A}
\calA_{L/\Q}= 
\begin{cases}
\Z [G][{\bf e}_{\mathfrak f}]=\Z[G] & \text{if $p\nmid \mathfrak f $,} \\
\Z [G][{\bf e}_{\mathfrak f},\, {\bf e}_{\mathfrak f/p} ] &    \text{if $p | \mathfrak f$}.
\end{cases}
\end{equation}
From this result, we know that if $\eta$ is a generator of $\calA_{L/\Q}$-module $\calO_L$,
then there exists $u\in \calA_{L/\Q}^{\times}$ such that $\eta=u {\rm Tr}_{\Q(\zeta_{\mathfrak f})/L}(\zeta_{\mathfrak f})$
(resp. $\eta=u ({\rm Tr}_{\Q(\zeta_{\mathfrak f})/L}(\zeta_{\mathfrak f})+1)$) if $p\nmid \mathfrak f$
(resp. $p|\mathfrak f$), and there is a one-to-one correspondence between the set of all generators of $\calO_L$ and $\calA_{L/\Q}^{\times}$.
In this section, we consider the structure of the group $\calA_{L/\Q}^{\times}$.

First, we consider the group structure of $\Z [G]^{\times}$. Let $U_p =\{ u\in \Z [\zeta_p]^{\times} \, |\, u\equiv \pm 1 \pmod{ (1-\zeta_p)} \}$
and $u_k=(1-\zeta_p^k)/(1-\zeta_p)$ for $k\in \{1,2,\cdots,  (p-1)/2 \}$.
Let $\phi_p(X)=(X^p-1)/(X-1) \in \Z[X]$ be the $p$-th cyclotomic polynomial.
The following lemma was proven in \cite[Theorems~1.6 and 1.7]{AF} except for the injectivity.
\begin{lem}[\cite{AF}]\label{lem:AF}
A group homomorphism 
\[
\nu:\ \Z[G]^{\times} \longrightarrow \Z [\zeta_p]^{\times}, \quad \sigma \mapsto \zeta_p
\]
is injective, $\nu (\Z[G]^{\times})=U_p$ and
\[
\Z [\zeta_p]^{\times}/U_p= \{ \overline{u}_k \, |\, k\in  \{1,2,\ldots, (p-1)/2 \} \},
\]
where $\overline{u}_k=u_k U_p$.
\end{lem}
\begin{proof}
See \cite[Theorems~1.6 and 1.7]{AF} except for the injectivity. We show that $\nu$ is injective. Put
\[
\psi : (\Z[X]/(X^p-1))^{\times} \longrightarrow \Z [\zeta_p]^{\times}, \quad \overline{X} \mapsto \zeta_p.
\]
Since the composition of an isomorphism $\Z[G]^{\times} \overset{\sim}{\longrightarrow} (\Z[X]/(X^p-1))^{\times}$ and $\psi$ is $\nu$,
it suffices to show that $\psi$ is injective. Let $\overline{f} \in {\rm Ker}\, \psi$. Since $f(\zeta_p)=1,\, \phi_p=(X^p-1)/(X-1)\in  \Z[X],\, f-1 \in \Z[X]$ and
$\phi_p$ is monic, $\phi_p$ divides $f-1$ in $\Z[X]$. We have $f(1) \equiv 1\pmod{p}$. Furthermore, we have $f(1)=\pm 1$ from $\overline{f} \in 
(\Z[X]/(X^p-1))^{\times}$ and a ring homomorphism $\Z[X]/(X^p-1) \longrightarrow \Z,\ \overline{f} \mapsto f(1)$. From these facts and $p\ne 2$, we
obtain $f(1)=1$. Therefore, $X-1$ divides $f-1$ in $\Z[X]$. We conclude that $X^p-1=(X-1)\phi_p$ divides $f-1$ in $\Z[X]$, and $\overline{f} =1$ in
$(\Z[X]/(X^p-1))^{\times}$. We obtain that $\psi$ is injective.
\end{proof}

We will prove the following theorem on the structure of the group $\calA^{\times}_{L./\Q}$.
\begin{theo}\label{theo:A}
Let $p$ be an odd prime number and $L/\Q$ a cyclic extension of degree $p$. Let $\mathfrak f $ denote the conductor of $L$.
\begin{enumerate}
\item[(1)] If $p\nmid \mathfrak f$, then we have the following group isomorphism :
\[
\calA_{L/\Q}^{\times} \overset{\sim}{\longrightarrow} U_p,\quad \sigma \mapsto \zeta_p.
\]
\item[(2)] If $p|\mathfrak f$, then we have the following exact sequence of abelian groups :
\[
1 \longrightarrow \{ 1,\, 1-2{\bf e}_{\bf 1} \} \longrightarrow \calA_{L/\Q}^{\times} \overset{\psi}{\longrightarrow} \Z[\zeta_p]^{\times} \longrightarrow 1,
\]
where $\psi (\sigma)=\zeta_p$ and $\displaystyle{ {\bf e}_{\bf 1}=\frac{1}{p}\sum_{i=0}^{p-1} \sigma^i}$ is the idempotent for the trivial character ${\bf 1}$.
\end{enumerate}
\end{theo}
\begin{proof}
The assertion of (1) follows from (\ref{eq:A}) and Lemma~\ref{lem:AF}. We prove (2). First, we will prove that the map $\psi$ is surjective.
Let $\alpha \in \Z [\zeta_p]^{\times}$. By Lemma~\ref{lem:AF}, there exists $u\in U_p$ and $k\in \{1,2,\ldots, (p-1)/2 \}$ such that $\alpha =u_ku$.
Let $\beta \in \Z [\zeta_p]^{\times}$ satisfy $\alpha \beta =1$ and write $\beta =u_{\ell} v,\, v\in U_p$ and $\ell \in \{1,2,\ldots, (p-1)/2 \}$.
We can write $\alpha=f(\zeta_p)$ and $\beta=g(\zeta_p)$ for $f,g\in \Z[X]$. Since $f(\zeta_p )g(\zeta_p)=\alpha \beta =1$, we have
\begin{equation}\label{eq:fg=1}
fg \equiv 1 \pmod{(\phi_p) } \quad {\rm in}\ \Z[X].
\end{equation}
On the other hand, since $\alpha=u_ku \equiv \pm u_k \equiv \pm k \pmod{ (1-\zeta_p)}$ and 
$\beta=u_{\ell} v \equiv \pm u_{\ell} \equiv \pm \ell \pmod{(1-\zeta_p) }$, we obtain $f\equiv a \pmod{ (\phi_p,X-1)},\ a:=\pm k$ and
$g\equiv b \pmod{(\phi_p, X-1)},\ b:=\pm \ell$. Let $f_1, g_1 \in \Z[X]$ satisfy 
\begin{align}\label{eq:f,g}
& f\equiv  a+f_1 \phi_p \qquad \pmod{(X-1)}, \\
& g\equiv  \, b+g_1 \phi_p \qquad \pmod{(X-1)}. \notag
\end{align}
From (\ref{eq:fg=1}) and (\ref{eq:f,g}), we have $1\equiv fg\equiv ab \pmod{ (\phi_p,X-1)}$, and hence
$ab\equiv 1\pmod{p}$. Let $c\in \Z$ satisfy $ab=1+pc$. Define $f_2,g_2 \in \Q[X]$ by
\begin{align*}
& f_2=f-f_1\phi_p -\frac{a+1}{p} \, \phi_p,\\
& g_2=g-g_1 \phi_p-\frac{b+1}{p}\, \phi_p.
\end{align*}
By (\ref{eq:fg=1}), we have
\[
f_2g_2 \equiv fg \equiv 1 \pmod{ (\phi_p)} \quad {\rm in} \ \Q[X],
\]
and hence 
\begin{equation}\label{eq:Phi_p}
\phi_p \, | \, (f_2g_2-1) \qquad {\rm in} \ \Q[X].
\end{equation}
On the other hand, we have by (\ref{eq:f,g})
\[
f_2g_2 \equiv 1 \pmod{(X-1)} \qquad {\rm in} \ \Q[X],
\]
and hence
\begin{equation}\label{eq:X-1}
(X-1) \, |\, (f_2g_2-1) \quad {\rm in} \ \Q[X].
\end{equation}
From (\ref{eq:Phi_p}) and (\ref{eq:X-1}), we conclude that $X^p-1 =(X-1)\phi_p$ divides $f_2g_2-1$ in $\Q[X]$, and
$\overline{f}_2 \overline{g}_2 =1$ in $\Q[X]/(X^p-1)$. Define $y,z \in \calA_{L/\Q}$ by
\begin{align*}
& y=f(\sigma)-f_1(\sigma)p {\bf e}_{\bf 1} -(a+1){\bf e}_{\bf 1},\\
& z=g(\sigma)-g_1(\sigma)p{\bf e}_{\bf 1}-(b+1){\bf e}_{\bf 1}.
\end{align*}
The image of $yz$ by the map $\calA_{L/\Q} \hookrightarrow \Q[G] \simeq \Q[X]/(X^p-1),\ \sigma \mapsto \overline{X}$ is $\overline{f}_2\overline{g}_2=1$,
we obtain $yz=1$, and hence $y,z\in \calA_{L/\Q}^{\times}$.
Since $\psi ({\bf e}_{\bf 1})=0$, we have $\alpha=f(\zeta_p)=\psi (y),\ y\in \calA_{L/\Q}^{\times}$, and
hence $\alpha \in \psi (\calA_{L/\Q}^{\times})$. We obtain that $\psi$ is surjective.

Next, we  show that ${\rm Ker}\, \psi =\{ 1,1-2{\bf e}_{\bf 1} \}$. From (\ref{eq:A}) and ${\bf e}_{\mathfrak f}+{\bf e}_{\mathfrak f/p}=1,\ {\bf e}_{\mathfrak f/p}={\bf e}_{\bf 1}$,
we have
\[
\calA_{L/\Q}=\Z[G][{\bf e}_{\mathfrak f}, {\bf e}_{\mathfrak f/p}]=\Z [G][{\bf e}_{\bf 1}].
\]
Furthermore, since for any $x\in \Z[G]$ there exists $a\in \Z$ such that $x{\bf e}_{\bf 1}=a{\bf e}_{\bf 1}$ and ${\bf e}_{\bf 1}^2={\bf e}_{\bf 1}$, 
we obtain $\calA_{L/\Q}=\{x+a{\bf e}_{\bf 1}\, |\, x\in \Z[G],\ a\in \Z \}$. First, we show ${\rm Ker}\, \psi =\{1+a {\bf e}_{\bf 1} \, |\, a\in \Z \} \cap \calA_{L/\Q}^{\times}$.
Let $\widetilde{\psi} : \calA_{L/\Q}^{\times} \overset{\psi}{\longrightarrow} \Z[\zeta_p]^{\times} \overset{\sim}{\longrightarrow} (\Z[X]/(\phi_p) )^{\times}$.
It is enough to show ${\rm Ker}\, \widetilde{\psi} =\{1+a {\bf e}_{\bf 1} \, |\, a\in \Z \}\cap \calA_{L/\Q}^{\times}$.
We will prove ${\rm Ker}\, \widetilde{\psi} \subseteq \{1+a {\bf e}_{\bf 1} \, |\,  a\in \Z \} \cap \calA_{L/\Q}^{\times}$ since
 the
opposite inclusion is trivial. Let $\alpha=x+a{\bf e}_{\bf 1} \in {\rm Ker}\, \widetilde{\psi},\, x\in \Z[G],\, a\in \Z$.
we can write $x=f(\sigma)$ for $f(X)\in \Z[X]$. Since
\[
1=\widetilde{\psi} (\alpha)=\widetilde{\psi} (f(\sigma)+a{\bf e}_{\bf 1})=\overline{f(X)},
\]
we have $f(X)\equiv 1\pmod{(\phi_p)}$ in $\Z[X]$. Let $g\in \Z[X]$ satisfy $f=1+g\Phi_p$, and $c\in \Z$ satisfy $g(\sigma)
{\bf e}_{\bf 1}=c{\bf e}_{\bf 1}$. Then we have
\[
\alpha=x+a{\bf e}_{\bf 1} =f(\sigma )+a{\bf e}_{\bf 1} =1+g(\sigma )p{\bf e}_{\bf 1}+a {\bf e}_{\bf 1} =1+(cp+a) {\bf e}_{\bf 1},
\]
and hence $\alpha \in  \{1+a{\bf e}_{\bf 1} \, |\, a\in \Z \} \cap \calA_{L/\Q}^{\times}$. We obtain 
${\rm Ker}\, \psi =\{1+a{\bf e}_{\bf 1} \, | \,a \in \Z \} \cap \calA_{L/\Q}^{\times}$, and
to show ${\rm Ker}\, \psi =\{ 1,1-2{\bf e}_{\bf 1} \}$, we show $\{ 1+a {\bf e}_{\bf 1} \, |\, a\in \Z \} \cap \calA_{L/\Q}^{\times}
=\{1,1-2{\bf e}_{\bf 1} \}$. We have $1-2 {\bf e}_{\bf 1} \in \{1+a{\bf e}_{\bf 1} \, | \, a\in \Z \} \cap \calA_{L/\Q}^{\times}$ since
$(1-2{\bf e}_{\bf 1})^2=1$. Conversely, let $1+a{\bf e}_{\bf 1} \in \calA_{L/\Q}^{\times},\ a\in \Z$.
We will show that $a\in \{0,-2\}$.  Since $\{1+a{\bf e}_{\bf 1} \, |\, a\in \Z \}\cap \calA_{L/\Q}^{\times}={\rm Ker}\, \psi$ is
a subgroup of $\calA_{L/\Q}^{\times}$, there exists $1+b{\bf e}_{\bf 1} \in \calA_{L/\Q}^{\times},\ b\in \Z$ satisfying
\[
1=(1+a{\bf e}_{\bf 1})(1+b{\bf e}_{\bf 1})=1+ (ab+a+b) {\bf e}_{\bf 1}.
\]
From this equality, we obtain $ab+a+b=0$. The pair $(a,b)=(0,0)$ satisfies $ab+a+b=0$. We assume that $(a,b) \ne (0,0)$. Since
$a(1+b)=-b$, we have $b|a$. Let $t\in \Z$ satisfy $a=bt$. Then we have $t(1+b)=-1$. Since $t,b \in \Z$ and $b\ne 0$, we
conclude that $b=-2$ and $t=1$, and hence $(a,b)=(-2,-2)$. We obtain ${\rm Ker}\, \psi=\{1,\, 1-2{\bf e}_{\bf 1} \}$.
\end{proof}
\begin{cor}\label{cor:p=3}
If $p=3$, then we have
\[
\calA_{L/\Q}^{\times} =\begin{cases}
\langle -1 \rangle \times \langle \sigma \rangle =\{\pm1, \pm \sigma, \pm \sigma^2 \} & {\rm if}\ 3\nmid \mathfrak f, \\
\langle1-2{\bf e}_{\bf 1} \rangle \times \langle -1 \rangle \times \langle \sigma \rangle = \{ \pm 1, \pm \sigma, \pm \sigma^2 , \pm (1-2{\bf e}_{\bf 1}), \pm \sigma
(1-2{\bf e}_{\bf 1}), \pm \sigma^2 (1-2{\bf e}_{\bf 1}) \} &  {\rm if}\ 3 | \mathfrak f.
\end{cases}
\]
\end{cor}
\begin{proof}
The assertion follows from Theorem~\ref{theo:A} and $\Z[\zeta_3]^{\times}=U_3=\{\pm1, \pm\zeta_3, \pm \zeta_3^2\}$.
\end{proof}
\section{Proof of the theorem}\label{sec:proof}

In this section, we give the proof of Theorem~\ref{theo:main}. 
First, we have $n_1, n_2 \in \Z $ since $M\equiv N \pmod{2}$ from $4\mathfrak f =M^2+27N^2$.
We  show that $n_1$ and $n_2$ are coprime. Let $p$ be a prime number which divides both $n_1$ and $n_2$. Since $4\mathfrak f=M^2+27N^2$ and $2\nmid \mathfrak f$,
either  $n_1$ or $n_2$ is not divisible by $2$ (note that $M\equiv N \equiv 1\pmod{2}$ or $M \not\equiv N \pmod{4}$ holds), and  hence $p\ne 2$. Furthermore, 
we have $p\ne 3$ since $n_2=N \not\equiv 0 \pmod{3}$. Since $p$ divides both $n_1=(M-3N)/2$ and $n_2=N$, it follows that $p$ divides $M$. This is a contradiction 
since $4\mathfrak f=M^2+27N^2$, $\mathfrak f/9$ is square-free and $p\ne 2,3$. Therefore, $n_1$ and $n_2$ are coprime.
Since $4\mathfrak f=M^2+27N^2=4(n_1^2+3n_1n_2+9n_2^2)$, we have $\mathfrak f=n_1^2+3n_2n_2+9n_2^2$, and hence
$3|n_1$. It follows that  $f_n (X)$ for $n=n_1/n_2$ is irreducible over $\Q$ by Lemma~\ref{lem:irre}.
Let $t=n_1/3 \in \Z$. From $n_1=(M-3N)/2$ and $M=3M_0$, we have $2t+n_2=M_0\equiv 2 \pmod{3}$. and hence $t\not\equiv n_2 \pmod{3}$. It follows that 
the conductor of $L_n$ is $\mathfrak f$  (\cite[Corollary~1]{A1} and $D_{L_n}=\mathfrak f^2$ where $D_{L_n}$ is the discriminant of $L_n$).
We have  already proved (1), (2), (3) of the theorem.

Next, we prove the remaining (4) and (5).  From Lemma~\ref{lem:Aoki},
$\alpha+1$
is a generator of $\calO_{L_n}$ over $\calA_{L_n/\Q}$ for $\alpha :=n_2\rho_n-n_1/3 \in {\bf e}_{\mathfrak f} \calO_{L_n}$.
Let $\alpha'=\sigma (\alpha)$ and $\alpha''=\sigma^2(\alpha)$. 
Since 
\[
\calA_{L_n/\Q}^{\times}= \{ \pm 1, \pm \sigma, \pm \sigma^2 , \pm (1-2{\bf e}_{\bf 1}), \pm \sigma
(1-2{\bf e}_{\bf 1}), \pm \sigma^2 (1-2{\bf e}_{\bf 1}) \}
\]
from Corollary~\ref{cor:p=3}, there are 12 generators of $\calO_{L_n}$ over $\calA_{L_n/\Q}$ which are given by $u (\alpha +1)$ for $u\in \calA^{\times}_{L_n/\Q}$.
Since ${\bf e}_{\mathfrak f}\, {\bf e}_{\bf 1}=0$ and $\alpha \in {\bf e}_{\mathfrak f}\, \calO_{L_n}$, we have $(1-2{\bf e}_{\bf 1})(\alpha+1)=\alpha-1$.
Therefore, the 12 generators of $\calO_{L_n}$ are
\begin{equation}\label{eq:gens}
\pm \alpha \pm 1, \quad \pm \alpha' \pm 1, \quad \pm \alpha'' \pm1  \quad \text{(any double sign)}.
\end{equation}

On the other hand, from (\ref{eq:O_L}), we know that $\eta_0+1$ is a generator where $\eta_0={\rm Tr}_{\Q (\zeta_{\mathfrak f})/L_n)} (\zeta_{\mathfrak f})$, 
and hence $\eta_0+1$ is equal to one of the 12 generators of (\ref{eq:gens}). Since ${\bf e}_{\mathfrak f}\, \eta_0=\eta_0,\ {\bf e}_{\mathfrak f}\,1=0$,
$\{ \eta_0, \eta_1, \eta_2 \}$ must be $\{ \alpha, \alpha', \alpha'' \}$ or $\{ -\alpha, -\alpha', -\alpha'' \}$.
Let
$h(X) =(X-\alpha)(X-\alpha')(X-\alpha'')$
be the minimal polynomial of $\alpha$.
Since $\rho_n, \rho_n'$ and $\rho_n''$ are roots of $f_n(X)$
and $\alpha=n_2\rho_n-n_1/3$,
we have
\begin{equation}\label{eq:fn-h}
n_2^3 f_n(X)=h\left( n_2X-\frac{n_1}{3} \right).
\end{equation}
Let  $P(X)=(X-\eta_0)(X-\eta_1)(X-\eta_2)$ be the period polynomial.
Since $\{ \eta_0,\eta_1,\eta_2 \} =\{ \alpha, \alpha',\alpha'' \}$ or $\{-\alpha, -\alpha', -\alpha'' \}$
and the coefficients of $x^2$ are zero, 
the difference of two polynomials $h(X)$ and $P(X)$ is only the sign of the constant term. To determine the sign, we calculate the values of
$\eta_0 \eta_1 \eta_2$ and $\alpha \alpha' \alpha''$ modulo $3$.
Let  $\mathscr{X}=\langle \chi \rangle$ be the group of Dirichlet characters associated to $L_n$
where $\chi=\chi_{3^2} \chi_{p_1}\cdots \chi_{p_{\nu}}$ and $\chi_m$ is the Dirichlet character of conductor $m$.
We define {\it the Gaussian sum} $\tau (\chi)$ for the character $\chi$ of conductor $\mathfrak f$ is
\[
\tau (\chi) =\sum_{a\in (\Z/\mathfrak f\Z)^{\times} }\chi(a) \zeta_{\mathfrak f}^a.
\]
We have $\tau (\chi)=\eta_0+\zeta_3 \eta_1+\zeta_3^2 \eta_2$ or $\eta_0 +\zeta_3^2 \eta_1+\zeta_3 \eta_2$  for the primitive third root of unity $\zeta_3=e^{2\pi i/3}$.
Since $\eta_0+\eta_1+\eta_2={\rm Tr}_{\Q (\zeta_{\mathfrak f})/L_n} (\zeta_{\mathfrak f}) =\mu (\mathfrak f)=0$, 
we have $\tau (\chi) +\overline{\tau (\chi)} =3\eta_0$, and hence we obtain
\begin{equation*}
(3\eta_0)^3=\tau(\chi)^3+\overline{\tau (\chi)}^3+9\eta_0 \tau(\chi)\, \overline{\tau (\chi)}.
\end{equation*}
Furthermore, since 
$\tau (\chi)\, \overline{\tau (\chi)}=\mathfrak f=3^2 p_1\cdots p_{\nu}$, it conclude that
\begin{equation}\label{eq:sum-Gauss}
(3\eta_0)^3=\tau(\chi)^3+\overline{\tau (\chi)}^3+3^4 p_1\cdots p_{\nu} \eta_0.
\end{equation}
Since $\chi=\chi_{3^2}\chi_{p_1}\cdots \chi_{p_{\nu}}$, we have 
\begin{equation}\label{eq:GS3}
\tau (\chi)^3=\tau (\chi_{3^2})^3 \tau (\chi_{p_1})^3\cdots \tau (\chi_{p_{\nu}})^3.
\end{equation}
By direct calculation, we have
\begin{equation}\label{eq:GS-cond9}
\tau(\chi_{3^2}) =\begin{cases}
27 \zeta_3  &\text{if $\chi_{3^2}(a)=\zeta_3^{(a^2-1)/3}$,} \\
27 \zeta_3^{-1}  &\text{if $\chi_{3^2}(a)=\zeta_3^{-(a^2-1)/3}$,} 
\end{cases} 
\end{equation}
\begin{equation}\label{eq:GS-pi}
\tau (\chi_{p_i})=\sum_{a\in (\Z/p_i\Z)^{\times}} \chi_{p_i}(a) \zeta_{p_i}^a \equiv -1 \pmod{(1-\zeta_3)} \qquad \text{in}\ \Z [\zeta_{\mathfrak f}]
\end{equation}
for $i\in \{1,\ldots ,\nu \}$. From (\ref{eq:sum-Gauss}), (\ref{eq:GS3}), (\ref{eq:GS-cond9}) and (\ref{eq:GS-pi}), we obtain
\begin{align*}
(3\eta_0)^3 & \equiv 2\times (-1)^{\nu} \times 27+3^4p_1\cdots p_{\nu} \eta_0 \\
& \equiv 2\times (-1)^{\nu} \times 27 \pmod{27(1-\zeta_3)} \qquad \text{in}\ \Z [\zeta_{\mathfrak f}].
\end{align*}
We conclude that $\eta_0^3 \equiv 2\times (-1)^{\nu} \equiv (-1)^{\nu+1} \pmod{(1-\zeta_3)}$ and hence
$(\eta_0 \eta_1 \eta_2)^3\equiv (-1)^{\nu+1} \pmod{(1-\zeta_3)}$.
Since $\eta_0\eta_1\eta_2 \in \Z$, we have $(\eta_0 \eta_1\eta_3)^3 \equiv (-1)^{\nu+1} \pmod{3}$ and hence 
\begin{equation}\label{eq:3etas}
\eta_0\eta_1\eta_2 \equiv (-1)^{\nu+1} \pmod{3}.
\end{equation}
On the other hand, we obtain 
\begin{align}
\alpha \alpha' \alpha'' & =\left(n_2\rho_n-\frac{n_1}{3} \right)\left(n_2\rho_n'-\frac{n_1}{3}\right)\left(n_2\rho_n'' -\frac{n_1}{3}\right) \label{eq:3alphas} \\
& =\frac{1}{27} (n_1^2+3n_1n_2+9n_2^2) (2n_1+3n_2)  \notag \\
& =\frac{1}{27} \mathfrak fM =p_1\cdots p_{\nu} M_0 \equiv -1 \pmod{3}. \notag
\end{align}
From (\ref{eq:3etas}) and (\ref{eq:3alphas}), we have
\[
\{ \eta_0, \eta_1, \eta_2 \} =\{ (-1)^{\nu} \alpha, (-1)^{\nu} \alpha', (-1)^{\nu} \alpha'' \} =\{\mu (\mathfrak f/9)\alpha,
\mu (\mathfrak f/9)\alpha', \mu (\mathfrak f/9)\alpha'' \}.
\]
Therefore, we have
\begin{equation}\label{eq:h-P}
h(X)=\mu (\mathfrak f/9) P(\mu (\mathfrak f/9) X).
\end{equation}
From (\ref{eq:fn-h}) and (\ref{eq:h-P}), we obtain (4) of the theorem, and (5) follows from (4).

Finally, let $(M',N')$ be another pair satisfying (\ref{eq:4f}) and put $n_1'=(M'-3N')/2$ and $n_2'=N'$.
Since $2n_1+3n_2=M \ne M' =2n_1'+3n_2'$, we have $L_n \ne L_{n'}$ by Lemma~\ref{lem:not-L}.
Since there are exactly $2^{\nu -1}$ pairs $(M,N)$ which satisfy (\ref{eq:4f}) (\cite[p.\. 342, 343]{Co}) and there are exactly $2^{\nu-1}$ cubic fields with conductor $\mathfrak f$,
any cyclic cubic field with conductor $\mathfrak f$ must coincide with $L_n$ for $n=n_1/n_2$ where $n_1$ and $n_2$ are defined by such a pair $(M,N)$.
%
%
\section{Examples}\label{sec:ex}
We consider cyclic cubic fields with conductor $\mathfrak f=9\times 7\times 13$.
All the pairs $(M,N)$ satisfying (\ref{eq:4f}) are $(M,N)=(-3\cdot 19,1), (3\cdot 17, 5), (3\cdot 8, 10), (-3,11)$, and the corresponding pair $(n_1,n_2)$ are
$(-30,1), (18,5), (-3,10), (-18,11)$ in order.
Table~\ref{table} shows the Gaussian periods of $L_n$,  Shanks' cubic polynomials, and the period polynomials  for each $n=n_1/n_2$.
Put $\rho=\rho_n$.
 \begin{table}[h]
\caption{$\mathfrak f=9\times 7\times  13$}
{\small 
 \begin{center}
 \begin{tabular}{ |  >{\centering} p{5em}|>{\centering} p{15em}|>{\centering} p{10em}|c|}
 \hline
\rule{0pt}{5mm}   $(n_1,n_2)$  & $\{ \eta_0, \eta_1,\eta_2 \}$  & $f_n(X) $ &  \hspace{5mm} $P(X)$  \hspace{5mm} \\ 
\hline  \hline 
\rule{0pt}{5mm}   $(-30,1)$ & $\{ \rho+10,\,  \rho'+10,\, \rho''+10\}$ & $X^3+30X^2+27X-1$  &  $X^3-273X+1729$ \\ 
\hline
\rule{0pt}{5mm}   $(18,5)$ & $\{ 5\rho-6,\, 5\rho'-6,\, 5\rho''-6\}$ & $X^3-\frac{18}{5} X^2-\frac{33}{5} X-1$  &  $X^3-273 X-1547$ \\ 
\hline
\rule{0pt}{5mm}   $(-3,10)$ & $\{ 10\rho+1,\, 10\rho'+1,\, 10\rho''+1\}$ & $X^3+\frac{3}{10} X^2-\frac{27}{10} X-1$  &  $X^3-273 X-728$ \\ 
\hline
\rule{0pt}{5mm}   $(-18,11)$ & $\{ 11\rho+6,\, 11\rho'+6,\, 11\rho''+6\}$ & $X^3+\frac{18}{11} X^2 -\frac{15}{11} X-1$  &  $X^3-273X+91$ \\ 
\hline
\end{tabular}
\label{table}
\end{center}
}
\end{table}

\end{document}